\documentclass[a4paper,leqno,12pt,draft]{amsart}
\usepackage{amsmath}
\usepackage{amssymb}
\usepackage{amsfonts}
\usepackage{amsthm}
\usepackage{enumerate}
\usepackage{comment}
\usepackage{cite}
\usepackage[dvips]{graphicx,color}%色付け
\usepackage{delarray}%行列式
\usepackage{ascmac}%枠を付ける
\usepackage{fancybox}%影を付けるなど
\usepackage{lineno}
\usepackage[hidelinks,draft=false]{hyperref}
\theoremstyle{plain}
\newtheorem{theorem}{Theorem}[section]

\newtheorem{proposition}[theorem]{Proposition}
\newtheorem{lemma}[theorem]{Lemma}
\newtheorem{remark}[theorem]{Remark}

\numberwithin{equation}{section}
\numberwithin{equation}{section}

\def\XXint#1#2#3{{\setbox0=\hbox{$#1{#2#3}{\int}$}
\vcenter{\hbox{$#2#3$}}\kern-.5\wd0}}

\begin{document}
\title[Finite Morse index solutions in two dimensions]{Uniformly boundedness of finite Morse index solutions to semilinear elliptic equations with rapidly growing nonlinearities in two dimensions}
\author{Kenta Kumagai}
\address{Department of Mathematics, Tokyo Institute of Technology}
\thanks{This work was supported by JSPS KAKENHI Grant Number 23KJ0949}
\email{kumagai.k.ah@m.titech.ac.jp}
\date{\today}

\begin{abstract}
We consider the Gelfand problem with rapidly growing nonlinearities in two-dimensional bounded strictly convex domains. In this paper, we prove the uniformly boundedness of finite Morse index solutions. As a result, we show that
there exists a solution curve having infinitely many bifurcation points or turning points. These results are recently proved by the present author \cite{kuma2025} for supercritical nonlinearities when the domain is the unit ball via an ODE argument. Instead of the ODE argument, we apply a new method focusing on the interaction between the growth condition of the nonlinearities and the shape of the fundamental solution of the Laplace equation. As a consequence, we clarify the bifurcation structure for general convex domains. 

\end{abstract}

\maketitle

\raggedbottom

\section{Introduction}
Let $\Omega\subset \mathbb{R}^2$ be a bounded strictly convex domain of class $C^2$. We consider the elliptic equation
\begin{equation}
\label{gelfand}
\left\{
\begin{alignedat}{4}
 -\Delta u&=\lambda f(u)&\hspace{2mm} &\text{in } \Omega,\\
u&=0  & &\text{on } \partial \Omega.
\end{alignedat}
\right.
\end{equation}
We assume that $f$ satisfies
\begin{equation}
\label{asf1}
\text{$f\in C^2[0,\infty)$ such that $f(0)>0$, 
$f'\ge 0$ and $f''\ge 0$;} 
\tag{$f_1$}
\end{equation}
\begin{equation}
\label{asf2}
(e+f(u))\log (e+f(u))\le M f'(u)\hspace{2mm}\text{for all $u>u_0$ with  $M>0$ and $u_0>0$.}
\tag{$f_2$}
\end{equation}
The condition \eqref{asf2} implies that $f(u)\ge \exp(c\exp(u/M))-e$ with some $c>0$ for any $u>u_0$ (see Lemma \ref{grlem}). The following are examples of $f$ satisfying \eqref{asf2}:
\begin{equation*}
    f(u)=\exp(\exp(u/q)), f(u)=\exp(\exp (u^p)), \hspace{1mm}\text{and}\hspace{1mm} 
    f(u)=\exp(\exp(\cdots\exp (u)\cdots),
\end{equation*}
where $q>0$ and $p\ge 1$.

We are interested in the relation between the Morse index and the $L^{\infty}$ norm of the solution $u$. For a given $\Omega'\subset\subset\Omega$ and $k\in \mathbb{N}$, we say that a solution $u$ of \eqref{gelfand} has the Morse index $k$ in $\Omega'$, and we denote by $m(u,\Omega')=k$, if $k$ is the maximal dimension of a subspace $X_k\subset C^{1}_{c}(\Omega')$ such that 
\begin{equation*}
    Q_{u}(\xi):=\int_{\Omega'}\left(|\nabla\xi|^2-\lambda f'(u)\xi^2\right)\,dx<0 \hspace{4mm}\text{for any $\xi\in X_k \setminus \{0\}$}.
\end{equation*}
For the corresponding problem in $\Omega \subset \mathbb{R}^N$ with $N\ge 3$, a celebrated result of Bahri and Lions \cite{BL} states that the $L^{\infty}$ norms of solutions are bounded if and only if their Morse indices are bounded when $f$ belongs to a class of subcritical nonlinearities including $f(u)=(1+u)^p$ with $1<p<\frac{N+2}{N-2}$, and $\lambda_1<\lambda<\lambda_2$ with some $0<\lambda_1<\lambda_2$. We refer to related papers \cite{Ha,Y} dealing with variants of the subcritical case. Here, we mention that 
bounds of $\lambda$ from both above and below are necessary to obtain this property. Indeed, it is known in \cite{JL} that when $\Omega$ is the unit ball $B_1$ and $f(u)=(1+u)^{p}$ with $p<\frac{N+2}{N-2}$, there exists a family of solutions $\{u_{\lambda}\}_{0<\lambda<\lambda^{*}}$ of \eqref{gelfand} with $\lambda^{*}\in (0,\infty)$ such that $m(u_{\lambda}, B_1)=1$ and $\lVert u_{\lambda}\rVert_{L^{\infty}(B_1)}\to \infty$ as $\lambda\to 0$. 

For the supercritical case, in contrast to the subcritical case, the dimension $N$ plays a key role in the uniformly boundedness of finite Morse index solutions. Indeed, when $N\ge 10$,  
it is known in \cite{JL} that there exists a family of solutions $\{u_{\lambda}\}_{0<\lambda<\lambda^{*}}$ with $\lambda^{*}\in (0,\infty)$ such that $m(u_\lambda,B_1)=0$ and 
$\lVert u_{\lambda}\rVert_{L^{\infty}(B_1)}\to\infty$ as $\lambda\to\lambda^{*}$ in the case $\Omega=B_1$ and $f(u)=e^u$. It implies that the uniformly boundedness result is not satisfied for $N\ge 10$ even when $\lambda$ is bounded from both above and below. On the other hand, when $3\le N\le 9$, the uniformly boundedness result is proved by \cite{Fa,Dn,Dn13,Dn8,Dn08} for power-type nonlinearities and exponential-type nonlinearities.
Then, their results are extended by Figalli and Zhang \cite{Figalli} to all supercritical nonlinearities with $3\le N\le 9$. Here, we point out that the authors \cite{Figalli} proved it without imposing any lower bounds of $\lambda$. 

In two-dimensional case, unlike the case $3\le N\le 9$, uniformly boundedness of finite Morse index solutions is not satisfied without imposing a lower bound of $\lambda$ even for the exponential nonlinearity. Indeed, it is known \cite{JL} that when $f(u)=e^u$ and $\Omega=B_1$,
there exists a family of solutions $\{u_{\lambda}\}_{0<\lambda<\lambda^{*}}$ of \eqref{gelfand} with some $\lambda^{*}\in (0,\infty)$ such that $m(u_{\lambda}, B_1)= 1$ and $\lVert u_{\lambda}\rVert_{L^{\infty}(B_1)}\to \infty$ as $\lambda\to 0$.
On the other hand, very recently, the present author \cite{kuma2025} proved the uniformly boundedness result for supercritical nonlinearities without imposing any lower bounds in the case $\Omega=B_1$. More precisely, the author \cite{kuma2025} proved that if $\Omega=B_1$ and 
$f$ satisfies
\begin{equation}
\label{condsuper}
\limsup_{u\to\infty}\left(\frac{F(u)\log F(u)}{f(u)}\right)'<\frac{1}{2} \hspace{2mm} \text{with }F(u)=\int_{0}^{u}f(s)\,ds,
\end{equation}
then the $L^{\infty}$ norms of the solutions to \eqref{gelfand} with is controlled by some $C>0$ depending only on $f$ and the supremum of their Morse indices. Here, we remark that $f=e^{u^p}$ with $p>2$ and $f=e^{e^u}$ satisfy \eqref{condsuper}. For more examples satisfying
\eqref{condsuper}, see \cite[Example 2.1]{kuma2025}. The key idea of \cite{kuma2025} is to obtain an $H^1$ estimate for finite Morse index solutions, and to combine it with a non-existence result of an unstable solution for any small $\lambda>0$. We mention that the non-existence result is derived from the combination of a technical ODE argument and the Pohozaev-type identity provided in \cite{tang}, both of which are only applicable in the case $\Omega=B_1$. Therefore, for general domains,
this approach does not work well; and there are no uniformly boundedness results in the literature for the supercritical case.

The aim of this paper is to prove the 
uniformly boundedness result for all two-dimensional strictly convex domains without imposing any lower bounds of $\lambda$. In order to prove the result, we obtain an $H^1$ estimate for finite Morse index solutions for general strictly convex domains. Moreover, we apply a new method focusing on the interaction between the shape of the fundamental solution of the Laplace equation and the growth condition \eqref{asf2}, instead of the method via the non-existence result. As a result, we prove the following

\begin{theorem}
\label{mainthm}
Let $k\in \mathbb{N}$, $\lambda^{*}>0$ and $\Omega\subset \mathbb{R}^2$ be a bounded strictly convex domain of class $C^2$. We assume that $f$ satisfies \eqref{asf1} and \eqref{asf2} with some $u_0>0$ and $M>0$. We assume in addition that there exists a nonnegative continuous function $g$ such that $f\le g$ for all $u>0$. Let $u$ be a solution of \eqref{gelfand}
for some $0<\lambda<\lambda^{*}$. Then, 
\begin{equation*}
\label{mainest}
    \lVert u \rVert_{L^{\infty}(\Omega)}\le C(\Omega, \lambda^{*},g,k, M, u_0).
\end{equation*}
\end{theorem}

%\begin{remark}
%\rm{As mentioned before, Theorem \ref{mainthm} is proved in \cite{kuma2025} for all nonlinearities satisfying \eqref{condsuper} when $\Omega=B_1$ by obtaining an $H^1$ estimate and a non-existence result. We mention that the $H^1$ estimate can be proved for all bounded strictly convex domains under the same assumption \eqref{condsuper} (see Proposition \ref{h1prop} and Remark \ref{imprem}). Moreover, in the proof of Theorem \ref{mainthm},
%the condition \eqref{asf2} is used only to prove the key estimate \eqref{green}. Here, we remark that a similar estimate appears in the proof of \cite[Proposition 3.3]{kuma2025}, which plays a key role to obtain the non-existence result in the case $\Omega=B_1$, and is proved by the Pohozaev-type identity obtained in \cite{tang}. From this observation, we believe that Theorem \ref{mainthm} is also satisfied without the assumption \eqref{asf2}, and a kind of Pohozaev-type identity plays a key role to prove
%\eqref{green}.}
%\end{remark}

Our result plays an important role to study the bifurcation structure of the Gelfand problem \eqref{gelfand}. Indeed, as a result of Theorem \ref{mainthm} and the analytic bifurcation theory provided in \cite[Section 2.1]{BD}, we obtain the following
\begin{theorem}
\label{mainthm2}
Let $\Omega\subset \mathbb{R}^2$ be a bounded strictly convex domain of class $C^2$. We assume that $f$ satisfies \eqref{asf1} and \eqref{asf2} with some $u_0>0$ and $M>0$. Moreover, we suppose that $f$ can be analytically extended to $\mathbb{R}$. Let
\begin{align}
\mathcal{S}:=\left\{(\lambda,u)\in\mathbb{R}_{+}\times C^{1}_{0}(\overline{\Omega}) : \, 
\begin{array}{ll}
-\Delta u=\lambda f(u)\hspace{2mm}\text{and  $-\Delta-\lambda f'(u)$ }\\ 
\text{is invertible with bounded inverse} 
\end{array}
\right\}. \notag
\end{align}
Then, there exists a continuous map
\begin{equation*}
 \mathbb{R}_{+}\ni s \mapsto h(s)=(\lambda(s) ,u(s))\in \overline{\mathcal{S}}\subset
 \mathbb{R}_{+}\times C^{1}_{0}(\overline{\Omega})
\end{equation*}
such that the following are satisfied.
\begin{enumerate}
    \item [{(i)}] 
    $\lambda(s)\to 0$ and $\lVert u(s)\rVert_{L^{\infty}(\Omega)}\to 0$ as $s\to 0$. Moreover, $m(u(s),\Omega)\to \infty$ and  $\lVert u(s) \rVert_{L^{\infty}(\Omega)}\to \infty$ as $s\to\infty$.
\item [{(ii)}] $h^{-1}(\overline{\mathcal{S}}\setminus \mathcal{S})$ is a discrete set. In addition, 
$h$ is injective in $h^{-1}(s)$ with $\lambda'(s)\neq 0$, and $h$ is analytic for each $s\in h^{-1}(\mathcal{S})$.
\item [{(iii)}] For each $s_0\in h^{-1}(\overline{\mathcal{S}}\setminus \mathcal{S})$, there exists a continuous and injective map $\eta:[-1,1]\to \mathbb{R}$ such that $\eta(0)=s_0$ and the following reparameterization
\begin{equation*}
    (-1,1)\ni t \to (\lambda(\eta(t)), u(\eta(t))) \in \overline{\mathcal{S}}
\end{equation*}
is analytic and the derivatives do not vanish except at $t=0$.
\item[{(iv)}] There exists a sequence $\{s_i\}_{i\in \mathbb{N}}\subset h^{-1}(\overline{\mathcal{S}}\setminus\mathcal{S})$ such that $s_i\to\infty$ and $h(s_i)$ is a  bifurcation point or a turning point for each $i$.
\end{enumerate}
Here, we say that $h(s_0)$ is a bifurcation point 
if there exists $\varepsilon>0$ such that
every neighborhood of $h(s)$ contains a solution of \eqref{gelfand} which is not in $h((s-\varepsilon, s+\varepsilon))$. In addition, we say that $h(s)$ is a turning  point if $h(s)$ is not a bifurcation point and $\lambda$ is not injective for each neighborhood of $s$.
\end{theorem}

This problem is extensively studied for a number of authors. We refer to \cite{Br, B, Bre, CFRS, Dup,cabres} for the study of the stable branch, and refer to \cite{korman, Mi2018, Mi2023, cabrecapella, Mi2024, kuma2025} for the global bifurcation curve in the case $\Omega=B_1$. 
In particular, when $2\le N\le 9$ and $\lim_{u\to\infty} f(u)/u=\infty $, it is known that there exists a stable branch emanating from $(0,0)$, going to $\lambda=\lambda^{*}\in (0,\infty)$, where no solution exists for $\lambda>\lambda^{*}$. Moreover, the curve $(\lambda(s), u(s))$ obtained in Theorem \ref{mainthm2} contains the stable branch and it turns at $\lambda=\lambda^{*}$.
 
In addition, when $3\le N\le 9$ and $\Omega$ is a convex domain, the existence of infinitely many bifurcation/turning points is proved in \cite{Fa,Dn,Dn13,Dn8,Dn08} for exponential-type nonlinearities and power-type nonlinearities. 
These results are obtained by proving the uniformly boundedness of finite Morse index solutions with a lower bound of $\lambda$, and then using the non-existence results \cite{DS,Sch} of an unstable solution with any small $\lambda>0$. Here, the non-existence results \cite{DS, Sch} are proved by essentially using the principal part $(N-2)\lVert \nabla u \rVert^2_{L^2(\Omega)}$ of the Pohozaev identity. However, in two-dimensional case, since the principal part vanishes, it is a non-trivial problem whether the non-existence result is satisfied for general domains, and it is proved in \cite{kuma2025} only for the case $\Omega=B_1$ by using a technical ODE argument.

Instead of obtaining the non-existence result, we succeed in directly proving the uniformly boundedness of finite Morse index solutions  without any lower bounds of $\lambda$. For the supercritical case, it is the first result clarifying the bifurcation structure for general domains in two dimensions to the best of the author's knowledge.

\section{Proof of Theorem \ref{mainthm} and Theorem \ref{mainthm2}}
The main purpose of this section is to prove Theorem \ref{mainthm} and Theorem \ref{mainthm2}. We start from proving the following
\begin{lemma}
\label{grlem}
We assume that $f$ satisfies \eqref{asf1} and \eqref{asf2} with some $M>0$ and $u_0>0$. Then, 
it follows that
\begin{equation*}
f(u)\ge e^{ce^{u/M}}-e \hspace{4mm} \text{for all $u>u_0$},   
\end{equation*}
where $c>0$ is a constant depending only on $u_0$ and $M$. 
\end{lemma}

\begin{proof}
The condition \eqref{asf2} can be rewritten as
\begin{equation*}
    \frac{1}{M}\le \frac{f'(u)}{(e+f(u))\log (e+f(u))}=\frac{d}{du}(\log(\log(e+f(u))).
\end{equation*}
Let $u>u_0$. By integrating the above on $(u_0, u)$, we obtain
\begin{equation*}
\frac{u-u_0}{M}\le \log((\log(e+f(u)))-\log(\log(e+f(u_0))).
\end{equation*}
Thus, the desired result follows.
\end{proof}

Next, we quote an $L^1$ estimate of $f'(u)$ for finite Morse index solutions.
\begin{lemma}[see \cite{Figalli, kuma2025}]
\label{basiclem}
Let $k\in\mathbb{N}$ and $\Omega\subset \mathbb{R}^2$ be a bounded 
domain. We assume that $f$ satisfies \eqref{asf1}. Let $u$ be a solution of \eqref{gelfand} satisfying $m(u,\Omega)\le k$.
Then, for any $B_{2r}(x)\subset\subset \Omega$, we have
\begin{equation*}
    \int_{B_r(x)} \lambda f'(u)\,dx \le C(k), 
\end{equation*}
where $C(k)$ is a constant depending only on $k$.
\end{lemma}

Using this lemma and the moving plane method, we obtain the following $H^1$ estimate of finite Morse index solutions.

\begin{proposition}
\label{h1prop}
Let $k\in \mathbb{N}$, $\lambda^{*}>0$ and $\Omega\subset \mathbb{R}^2$ be a bounded strictly convex domain of class $C^2$. Assume that $f$ satisfies \eqref{asf1} and \eqref{asf2} with some $u_0>0$ and $M>0$. We assume in addition that 
there exists a nonnegative function $g\in C^0(\mathbb{R})$ such that $f(t)\le g(t)$ for all $t\ge 0$. Let $u\in C^2(\Omega)$ be a solution of \eqref{gelfand} satisfying $m(u,\Omega)\le k$ and $\lambda<\lambda^{*}$. Then, we have 
\begin{equation*}
    \lVert u\rVert_{H^1(\Omega)}\le C(k, \Omega, M, g, u_0, \lambda^{*}).
\end{equation*}
\end{proposition}

\begin{proof}
Let $u$ be a solution of of \eqref{gelfand} satisfying $m(u,\Omega)\le k$. We choose $u_1>u_0$ depending only on $u_0$ and $M$ so that
\begin{equation*}
    \frac{u}{ce^{u/M}}\le 1 \hspace{4mm}\text{for all $u>u_1$},
\end{equation*}
where $c>0$ is that in Lemma \ref{grlem}. Then, it follows from \eqref{asf2} and Lemma \ref{grlem} that 
\begin{align*}
  \int_{\Omega}|\nabla u|^2\,dx&=   \int_{\Omega}\lambda f(u)u\,dx\\
  &\le \int_{\{u>u_1\}} \lambda M uf'(u)(\log (e+f(u)))^{-1}\,dx + \lambda^{*} g(u_1)u_1|\Omega|\\
  &\le\int_{\{u>u_1\}} \lambda M uf'(u)(ce^{u/M})^{-1}\,dx+ \lambda^{*} g(u_1)u_1|\Omega|\\
  &\le M\int_{\Omega}\lambda f'(u)\,dx+C(\Omega, g, u_0, \lambda^{*}, M).
\end{align*}
Let $d>0$ and we define $\Omega_{d}:=\{x\in \Omega:\mathrm{dist}(x,\partial\Omega)>d\}$.
Then, if $d$ is sufficiently small, we deduce from the moving plane method and the monotonicity of $f'$ that
\begin{equation*}
\int_{\Omega\setminus\Omega_d}\lambda f'(u)\,dx\le C(\Omega_d)\int_{\Omega_d} \lambda f'(u)\,dx.
\end{equation*}
We fix $d$ sufficiently small such that the above estimate is satisfied. Then, thanks to Lemma \ref{basiclem} and a covering argument, we have
\begin{equation*}
  \int_{\Omega_d} \lambda f'(u)\,dx\le C(k,\Omega_d). 
\end{equation*}
Therefore, by combining the above estimates and using the Poincar\'e inequality, we get the result.
\end{proof}
\begin{remark}
\label{imprem}
\rm{When $\Omega=B_1$, Proposition \ref{h1prop} is proved in \cite{kuma2025} for the nonlinearities satisfying \eqref{condsuper}. 
We can also prove Proposition \ref{h1prop} for the class of nonlinearities. Indeed, if $f$ satisfies \eqref{condsuper},
thanks to the statement and the proof of \cite[Lemma 2.2]{kuma2025}, we deduce that 
\begin{equation*}
    \frac{f(u)u}{F(u)\log F(u)}>2, \hspace{2mm}F(u)\ge e^{u^2} \hspace{2mm}\text{and}\hspace{2mm} \left(\frac{F(u)}{f(u)}\right)'=1-\frac{F(u)f'(u)}{f^2(u)}<0
\end{equation*}
for any $u$ sufficiently large. Therefore, it follows for any $u$ sufficiently large that
\begin{equation*}
2f(u)u\le f(u)u\cdot\frac{f(u)u}{F(u)\log F(u)}\le\frac{f^2(u)}{F(u)}\le f'(u).
\end{equation*}
As a result, by arguing the same as in the proof of Proposition \ref{h1prop}, we obtain the result.}
\end{remark}

Thanks to Proposition \ref{h1prop}, we have the following 
\begin{proposition}
\label{superprop}
Let $k\in \mathbb{N}$, $\lambda^{*}>0$ and $\Omega\subset \mathbb{R}^2$ be a bounded strictly convex domain with class $C^2$. Assume that $u_i$ is a sequence of solutions of \eqref{gelfand} with $f=f_i$ and $\lambda=\lambda_i\in (0,\lambda^{*})$
such that $m(u_i,\Omega)\le k$, where $f_i$
satisfies \eqref{asf1}, \eqref{asf2} with some $u_{0}>0$, $M>0$, and $f_i(t)\le g(t)$ for $t\ge 0$ with some non-negative function $g\in C^0(\mathbb{R})$. 
Then, there exist a subsequence $\{u_{i_j}\}_{j\in \mathbb{N}}\subset\{u_i\}_{i\in \mathbb{N}}$, a set $\Sigma\subset\Omega$ with $\mathrm{card}(\Sigma)\le k$ and a solution $u\in C^2(\Omega\setminus\Sigma) \cap H^1_{0}(\Omega)$ of \eqref{gelfand} with some $f\ge 0$ and $\lambda\in [0,\lambda^{*}]$ such that $u_{i_j}\rightharpoonup u$ in $H^1(\Omega)$,  $u_{i_j}\to u$ in $C^2_{\mathrm{loc}}(\Omega\setminus \Sigma)$ and $m(u,\Omega)\le k$. 
\end{proposition}
\begin{proof}
By Proposition \ref{h1prop}, we deduce that there exists a solution $u\in H^1_{0}(\Omega)$ such that $u_i\rightharpoonup u$ in $H^1(\Omega)$ by taking a subsequence if necessary. The left part is proved by using the same method as that of \cite[Proposition 2.3]{Figalli}.  
\end{proof}

\begin{proof}[Proof of Theorem \ref{mainthm}]
We assume by contradiction that there exists a sequence $\{u_i\}_{i\in \mathbb{N}}$ of solutions to \eqref{gelfand} with $f=f_i$ and 
$\lambda=\lambda_i$ such that $\lVert u_i \rVert_{L^{\infty}(\Omega)}\to \infty$ and $m(u_i,\Omega)\le k$, where $f_i$ and $\lambda=\lambda_i$ satisfy all the assumptions in the statement of Theorem \ref{mainthm} with some $M>0$, $u_0>0$, $0\le g\in C^{0}(\mathbb{R})$ and $\lambda^{*}>0$. 

Then, thanks to Proposition \ref{superprop}, we have
$u_i\to u$ in $C^2_{\mathrm{loc}}(\Omega\setminus \Sigma)$ with $\mathrm{card}(\Sigma)\le k$ and $u\in C^2(\Omega)$ by taking a subsequence if necessary. 
Without loss of generality, we assume that $\Sigma=\{x_0\}$. Then, it is sufficient to prove that 
\begin{equation}
\label{key}
    \lVert u_i\rVert_{L^\infty(B_{r}(x_0))}\le C \hspace{4mm} \text{with some $r>0$},
\end{equation}
where $C$ is independent of $i$.
Indeed, we verify that if $d>0$ is sufficiently small, we have $\sup_{\Omega_d} u_i=\sup_{\Omega}u_i$ by using the moving plane method, where $\Omega_{d}:=\{x\in \Omega:\mathrm{dist}(x,\partial\Omega)>d\}$. We fix $0<d<1$ so that the above assertion is satisfied. Then, by using the fact that $u_i\to u$ in $C^2_{\mathrm{loc}}(\Omega\setminus \Sigma)$ and \eqref{key}, we obtain $\lVert u_i\rVert_{L^{\infty}(\Omega)}= \lVert u_i\rVert_{L^{\infty}(\Omega_d)}\le C$, which is a contradiction.

Finally, we show \eqref{key}. In the case $x_0\in\Omega\setminus\Omega_{d/2}$, we fix $r=\frac{d}{4}$. Then, since $B_r(x_0)\subset\Omega\setminus\Omega_{d}$, it follows from the fact $u_i\to u$ in $C^2_{\mathrm{loc}}(\Omega\setminus \Sigma)$ that 
\begin{equation*}
    \lVert u_i \rVert_{L^{\infty}(B_r(x_0))}\le \lVert u_i \rVert_{L^{\infty}(\Omega\setminus\Omega_{d})}\le \lVert u_i\rVert_{L^{\infty}(\Omega_d)}\le C.
\end{equation*}
In the case $x_0\in\Omega_{d/2}$, we fix $r=\frac{d}{20}$. Then, for $y\in B_{r}(x_0)$, it follows from Green's representation formula that 
\begin{equation*}
2\pi u_i(y)=\int_{B_{2r}(y)}-\lambda_i f(u_i)\log|x-y|\,dx+\int_{\partial B_{2r}(y)} \left(\frac{u_i}{|x-y|}-\frac{\partial u_i}{\partial\nu}\log|x-y| \right) \,dS,
\end{equation*}
where $\nu$ is the outer unit normal to $\partial B_{2r}(y)$.
Thanks to the fact that $u_i\to u$ in $C^2_{\mathrm{loc}}(\Omega\setminus \Sigma)$ and $\partial B_{2r}(y)\subset \Omega_{7d/20}\setminus B_{r}(x_0)$, we have 
\begin{equation*}
   \int_{\partial B_{2r}(y)} \left(\frac{u_i}{|x-y|}-\frac{\partial u_i}{\partial\nu}\log|x-y| \right) \,dS
   \le C.
\end{equation*}
Therefore, it suffices to prove 
\begin{equation}
\label{green}
    \int_{B_{2r}(y)}-\lambda_i f(u_i)\log|x-y|\,dx\le C.\notag
\end{equation}
We denote by $s=|x-y|$. Here, we remark that $B_{4r}(y)\subset\subset\Omega_{d/4}$. Thus, by using Lemma \ref{basiclem}, it follows that
\begin{align*}
&\int_{B_{2r}(y)}-\lambda_i f(u_i)\log s\,dx\\
&\le \int_{\{ sf(u_i)\le 1\}\cap B_{2r}(y)} -  \lambda_i f(u_i)\log s\,dx+\int_{\{ sf(u_i)\ge 1\}\cap B_{2r}(y)} - \lambda_i f(u_i)\log s\,dx\\
&\le C+\int_{\{ sf(u_i)\ge 1\}\cap B_{2r}(y)} \lambda_i  f(u_i)\log f(u_i)\le C + \int_{B_{2r}(y)}M\lambda_i f'(u_i)\,dx\le C.
\end{align*}
Therefore, we get the result.
\end{proof}

\begin{proof}[Proof of Theorem \ref{mainthm2}]
We define 
\begin{equation*}
    \mathcal{O}=\{u\in C^{1}_{0}(\overline{\Omega}): u>0 \hspace{2mm}\text{in }\Omega \hspace{2mm} \text{and}\hspace{2mm} \frac{\partial u}{\partial \nu}<0 \hspace{2mm}\text{on } \partial \Omega\}
\end{equation*}
and
\begin{equation*}
\mathbb{R}_{+}\times \mathcal{O}\ni (\lambda,u)\to  \mathcal{F}((\lambda,u)):=u+\lambda \Delta^{-1}[f(u)] \in C^{1}_{0}(\overline{\Omega}), 
\end{equation*}
where $\nu$ is the outer unit normal to $\partial \Omega$ and $\Delta^{-1}$ is the inverse of the Dirichlet Laplacian in $\Omega$.
Then, thanks to Lemma \ref{grlem} and \cite[Proposition 3.3.1]{Dup}, we deduce that 
\eqref{gelfand} has no solutions for $\lambda>\lambda^{*}$ with some $\lambda^{*}\in (0,\infty)$. Thus, thanks to Theorem \ref{mainthm} and the analytic theory \cite[Section 2.1]{BD}, by the same argument as in the proof of \cite[Theorem 1]{Dan}, we have the result. 
\end{proof}

% ------------------------------------------------------------------------

\subsection*{Acknowledgment}
The author appreciates Professor Michiaki Onodera for his valuable advice in order to clarify the presentation. The author is also grateful to Yaqing Peng for pointing out an error of the previous version.

\bibliographystyle{plain}
\bibliography{morse_index_2.bib}

\end{document}